\documentclass[]{elsarticle}

\usepackage[margin=1.25in]{geometry} 
\usepackage[intlimits]{amsmath}
\usepackage{amsthm,amssymb,amsfonts} 
\usepackage{amssymb}
\usepackage{graphicx,xcolor,float} 
\usepackage{xargs} 
\usepackage{hyperref} 
\usepackage[shortlabels]{enumitem}
\usepackage[colorinlistoftodos,textsize=small,disable]{todonotes} 

\usetikzlibrary{decorations,decorations.pathmorphing,decorations.pathreplacing}

	\newcommand{\R}{\mathbb{R}}
	\newcommand{\N}{\mathbb{N}}
	\newcommand{\Q}{\mathbb{Q}}
	\newcommand{\T}{\mathbb{T}}
	\newcommand{\Z}{\mathbb{Z}}
	\newcommand{\C}{\mathbb{C}}
	
	\newcommand{\D}{\mathbb{D}}

	\newcommand{\ch}{\mathcal{H}}
  \newcommand{\cA}{\mathcal{A}}
  \newcommand{\cU}{\mathcal{U}}

	\newcommand{\cplus}{\mathbb{C}_{+}}

\numberwithin{equation}{section}

\newtheorem{prop}{Proposition}
\newtheorem{thm}[prop]{Theorem}
\newtheorem*{thm*}{Theorem}
\newtheorem*{mythma}{Theorem \ref{mainthm}\ref{hardway}}
\newtheorem*{mythmb}{Theorem \ref{mainthm}\ref{easyway}}
\newtheorem{lemma}{Lemma}

\newtheorem{remark}[prop]{Remark}

\begin{document}

\begin{frontmatter}

\title{Carlson's Theorem for Different Measures}



\author[mymainaddress]{Meredith Sargent\corref{mycorrespondingauthor}}
\cortext[mycorrespondingauthor]{Corresponding author}
\ead{meredithsargent@wustl.edu}

\address[mymainaddress]{Department of Mathematics, Campus Box 1146, Washington University in St Louis, St Louis, MO 63130}

\begin{abstract}
We use an observation of Bohr connecting Dirichlet series in the right half plane $\mathbb{C}_+$ to power series on the polydisk to interpret Carlson's theorem about integrals in the mean as a special case of the ergodic theorem by considering any vertical line in the half plane as an ergodic flow on the polytorus. Of particular interest is the imaginary axis because Carlson's theorem for Lebesgue measure does not hold there. In this note, we construct measures for which Carlson's theorem does hold on the imaginary axis for functions in the Dirichlet series analog of the disk algebra $\mathcal{A}(\mathbb{C}_+)$.
\end{abstract}

\begin{keyword}
Dirichlet Series, Ergodic Theorem
\end{keyword}

\end{frontmatter}

\section{Introduction} 
\label{sec:introduction}

  In 1913, Bohr~\cite{Bohr1913a} observed that one may connect Dirichlet series converging on the right half plane $\cplus = \{s \in \C:\; \Re(s)>0\}$ to power series on the infinite polydisk using the correspondence
    \begin{equation*}
      z_1 = 2^{-s},\,z_2 = 3^{-s},\, \dots\,,\, z_j = p_j^{-s}, \dots
    \end{equation*}
  where $p_j$ denotes the $j$th prime. For a Dirichlet series $f=\sum_{n=1}^\infty a_n n^{-s}$, we can use the fundamental theorem of arithmetic to factor each integer $n$ uniquely and then represent $f$ by a power series $F$ in the variables $\{z_j\}$.  As discussed in~\cite{SaksmanSeip2009}, the Bohr correspondence also allows us to consider any vertical line in $\C$ as an ergodic flow on the infinite-dimensional polytorus \(\T^\infty\):
    \begin{equation*}
      (e^{i\theta_1},\, e^{i\theta_2},\, \dots ) \mapsto (p_1^{-it} e^{i\theta_1},\, p_2^{-it}e^{i\theta_2},\, \dots) \in \T^\infty
    \end{equation*}
  and in particular, the imaginary axis maps to the boundary of the infinite polydisk (of radius one.) 
  We would like to compare a ``space average'' of the power series $F$ on $\T^\infty$ to a ``time average'' of the Dirichlet series $f$ on the ergodic flow described above. For this question, we consider $\ch^\infty$, the Banach space of Dirichlet series of the form 
    \begin{equation} \label{DS}
      f(s) = \sum_{n=1}^\infty a_n n^{-s}
    \end{equation} 
  that converge to bounded analytic functions on $\cplus$.

  A theorem of Carlson~\cite{Carlson1922} tells us about the limit in the mean of a Dirichlet series on an ergodic flow for $\sigma > 0$. 
    \begin{thm}[Carlson's Theorem] \label{thm:carlson} 
      If a Dirichlet series $f(s) = \sum_{n=1}^\infty a_n n^{-s}$ converges in the right half plane $\cplus$ and is bounded in every half plane $\Re(s)\geq\delta$ for $\delta > 0$, then for each $\sigma>0$
      \begin{equation}
        \lim_{T \to \infty} \frac{1}{T} \int_0^T |f(\sigma + it)|^2 dt = \sum_{n=1}^\infty |a_n|^2 n^{-2\sigma}
      \end{equation}
    \end{thm}

  Saksman and Seip showed in~\cite{SaksmanSeip2009} that Carlson's theorem fails to hold on the imaginary axis when we replace $f(\sigma+it)$ with its non-tangential limit $f(it)$ (which exists for almost every $t$.) 
  \begin{thm}[Saksman-Seip] \label{thm:SS}
    The following two statements hold:
    \begin{itemize}
      \item[(i)] There exists a function $f$ in $\ch^\infty$ such that 
        \begin{equation*}
          \lim_{T\to \infty} \frac{1}{T} \int_0^T |f(it)|^2 dt
        \end{equation*}
        does not exist.
      \item[(ii)] Given $\epsilon>0$, there exists a singular inner function $g = \sum_{n=1}^{\infty}b_n n^{-s}$ in $\ch^\infty$ such that $\sum_{n=1}^{\infty}|b_n|^2 \leq \epsilon$.
    \end{itemize}
  \end{thm}

  What this result tells us is that there can be no direct analog of Carlson's Theorem on the boundary: the limit need not exist and equality need not hold, at least not for Lebesgue measure and for all functions in $\ch^\infty$. However, by looking at a smaller space, we can prove an analog of Carlson's theorem.

  The space we consider is $\cA(\cplus)$, the set of Dirichlet series which are convergent on $\cplus$ and define uniformly continuous functions there. In~\cite{Aron2016}, Aron, Bayart, Gauthier, Maestre, and Nestoridis show that $\cA(\cplus)$ is a closed subspace of $\ch^\infty$ and prove that it consists exactly of the uniform limits of Dirichlet polynomials: 
    \begin{thm} \label{aron}
      Given $f: \cplus \to \C$ the following are equivalent.
      \begin{enumerate}[(i)]
        \item $f$ is the uniform limit on $\cplus$ of a sequence of Dirichlet polynomials.
        \item $f$ is represented by a Dirichlet series pointwise on $\cplus$ and $f$ is uniformly continuous on $\cplus$.
       \end{enumerate}
    \end{thm}

  We now come to the main result:
  \begin{thm} \label{mainthm}
    For a Dirichlet series $f(s)=\sum_{n=1}^\infty a_n n^{-s}$, let $F(z)$ be the corresponding power series on $\T^\infty$ under the Bohr lift.
    \begin{enumerate}[(i)]
      \item\label{hardway} Let \(\mu\) be a Borel probability measure on the infinite torus \(\T^{\infty}\). There exists a locally finite Borel measure \(\lambda\) on $\R$, such that, for all $f \in \cA(\cplus)$ 
      \begin{equation} \label{bigequality}
        \lim_{T \to \infty} \frac{1}{\lambda([0,T])} \int_{0}^{T}\left|f(it)\right|^2 d\lambda(t)
        = \int_{\T^{\infty}} \left| F(z) \right|^2 d\mu(z).
      \end{equation} 
      \item\label{easyway} Let $\lambda$ be a locally finite Borel measure on $\R$ such that the limit on the left hand side of (\ref{bigequality}) exists and is finite for all $f \in \cA(\cplus)$. Then there exists a unique Borel probability measure $\mu$ on the infinite torus \(\T^{\infty}\) such that, for all $f \in \cA(\cplus)$, (\ref{bigequality}) holds.
    \end{enumerate}
  \end{thm} 

  Part \ref{easyway} follows from the Riesz representation theorem and will be shown in Section \ref{secconverse}.
  To prove Theorem \ref{mainthm}\ref{hardway}, it is helpful to consider the following useful lemma which allows us to first consider linear combinations of point masses and construct corresponding measures $\lambda$ on $\R$, and then use that result to construct $\lambda$ for Borel probability measures $\mu$.
    \begin{lemma}[\cite{parthaProb} pp. 44-46] \label{dense}
      Let $X$ be a compact separable metric space with a dense subset $E$ and let $U(X)$ be the space of Borel probability measures on $X$. Then 
      \begin{enumerate}[(i)]
        \item The set of all measures whose supports are finite subsets of $E$ is dense in $U(X)$, with the weak-$\ast$ topology, and
        \item $U(X)$ is a compact metric space.
      \end{enumerate}
    \end{lemma}

\section{Construction of $\lambda$ for the Point Mass Case}  
    Before we construct $\lambda$, it is helpful to recall Kronecker's theorem:
    \begin{lemma}[Kronecker's Theorem]
      Let \( \phi_1, \dotsc , \phi_k \in \R\) be linearly independent over \(\Q\) and let \(\gamma_1, \dotsc , \gamma_k \in \R\) and \(T,\epsilon >0\) be given. Then there exists \(t>T\) and \(q_1, \dotsc , q_k \in \Z\) such that \[|t \phi_j - \gamma_j -q_j|<\epsilon, \, 1\leq j \leq k.\]
    \end{lemma}
    Kronecker's theorem will allow us to form $\lambda$ by placing point masses on $\R$ in a strategic way.
    \begin{lemma}[Linear combination of point masses]\label{lincombo}
      For every measure on $\T^{\infty}$ of the form $d\mu = \sum_{j=1}^N c_j \delta_{\omega_j}$ with $\sum_{j=1}^N c_j = 1$, there exists an infinite measure on $\R$, $\lambda$, such that
      \begin{equation}\label{nugood}
        \int_{\T^{\infty}}|F(z)|^2d\mu = \lim_{T \to \infty} \frac{1}{\lambda([0,T])} \int_{0}^{T} |f(it)|^2 d\lambda
      \end{equation}
      for Dirichlet polynomials $f$.
    \end{lemma}
      \begin{proof}
        Let $F(z) = \sum a_{\alpha} z^{\alpha}$ be a polynomial on $\T^\infty$ and let $f(it)= \sum a_\alpha (p_1^{\alpha_1}\cdots p_d^{\alpha_d})^{-it}$ be the corresponding Dirichlet polynomial. Note that because these have finitely many terms, there is some $d\in \N$ such that every $\alpha$ that appears is of the form $\alpha = (\alpha_1, \alpha_2, \dots, \alpha_d, 0, \dots)$.

        We will construct $\lambda$ to be a sum of point masses $t\in\R$, using Kronecker's theorem to place them so that their images under the Bohr lift $z\in \T^\infty$ approximate the point masses that make up $\mu$. In particular, we would like the images $z$ to fall within $\delta$-balls of $\T^\infty$ where given $\epsilon>0$, $\delta$ is chosen to be small enough such that if $|\omega-z|<\delta$, then
        \begin{equation} \label{continuitything}
          \left| |F(w)|^2-|f(it)|^2 \right| = \left| |F(w)|^2-|F(z)|^2 \right|< \epsilon.
        \end{equation}
        The first equality is because of the Bohr lift, and then we use the continuity of $F$.

        Let us examine $|F(w)|^2$ and $|f(it)|^2$: 
          \begin{align} \label{bigF}
            |F(\omega)|^2 &= \left| \sum a_{\alpha} \omega^\alpha \right|^2 \nonumber\\
                        &= \sum \left| a_{\alpha} \right|^2 \left|\omega^\alpha\right|^2 
                          + \sum_\alpha\sum_{\beta \neq \alpha} a_\alpha \overline{a_\beta} \omega^\alpha \overline{\omega^\beta}\nonumber\\
                        &= \sum \left| a_{\alpha} \right|^2  
                          + \sum_\alpha\sum_{\beta \neq \alpha} a_\alpha \overline{a_\beta} \omega^\alpha \overline{\omega^\beta}.
          \end{align}
        Now, expanding $|f(it)|^2$,
          \begin{align} \label{littlef}
            |f(it)|^2 &= \left| \sum a_\alpha (p_1^{\alpha_1}\cdots p_d^{\alpha_d})^{-it}\right|^2 \nonumber\\
                            &= \sum \left| a_{\alpha} \right|^2 
                              + \sum_\alpha\sum_{\beta \neq \alpha} a_\alpha \overline{a_\beta} (p_1^{\alpha_1}\cdots p_d^{\alpha_d})^{-it} \overline{(p_1^{\beta_1}\cdots p_d^{\beta_d})^{-it}} \nonumber\\
                            &= \sum \left| a_{\alpha} \right|^2  
                              + \sum_\alpha\sum_{\beta \neq \alpha} a_\alpha \overline{a_\beta} (p_1^{\alpha_1-\beta_1}\cdots p_d^{\alpha_d-\beta_d})^{-it}.
          \end{align}
        So we want to place point masses $t$ so that $(p_1^{\alpha_1-\beta_1}\cdots p_d^{\alpha_d-\beta_d})^{-it}$ is near $\omega^\alpha \overline{\omega^\beta}$ for all $\alpha$, $\beta$. Examine both sides. Since $\omega \in \T^{\infty}$, $\omega^\alpha = \omega_1^{\alpha_1} \omega_2^{\alpha_2} \cdots \omega_d^{\alpha_d}$ and there are $\theta_1,\, \theta_2, \dots , \theta_d$ so that 
          \begin{align*}
            \omega^\alpha \overline{\omega^\beta} 
                &= \omega_1^{\alpha_1} \omega_2^{\alpha_2} \dotsm \omega_d^{\alpha_d} \overline{\omega_1^{\beta_1} \omega_2^{\beta_2} \dotsm \omega_d^{\beta_d}}\\
                &= e^{i\theta_1 \alpha_1} e^{i\theta_2 \alpha_2} \dotsm e^{i\theta_d \alpha_d} e^{-i\theta_1 \beta_1} e^{-i\theta_2 \beta_2} \dotsm e^{-i\theta_d \beta_d}\\
                &= e^{i\theta_1 (\alpha_1-\beta_1)} e^{i\theta_2 (\alpha_2-\beta_2)} \dotsm e^{i\theta_d (\alpha_d-\beta_d)}.
            \intertext{On the other side, we have}
            \left(p_1^{\alpha_1-\beta_1}\dotsm p_d^{\alpha_d-\beta_d}\right)^{-it} 
              &= e^{-it \log \left(p_1^{\alpha_1-\beta_1}\dotsm p_d^{\alpha_d-\beta_d}\right)}\\
              &= e^{-i(\alpha_1-\beta_1)t\log(p_1)} \dotsm e^{-i(\alpha_d-\beta_d)t\log(p_d)}.
          \end{align*}
        Note that both of these lie on the unit circle, so the problem reduces to finding $t$ so that $$-t \log p_r \approx \theta_r^j \, mod\, 2\pi.$$ We can use Kronecker's theorem to do this, however, because Kronecker's theorem only holds for a finite collection, it can only be done for finitely many primes. Because we want the measure $\lambda$ to be independent of which primes appear in a polynomial, we will construct the measure in steps, so that the point masses farther from zero approximate the $\omega_j$ more accurately for more primes. This way any prime that appears in a Dirichlet polynomial will appear in our approximation at some level. 

        We also want the error from the poor approximations near zero to be small compared to the measure, so it will become irrelevant when we take the limit in the mean. This means that $\lambda$ needs to have the property that the measures of intervals far from zero are much larger than the measures of intervals nearer to zero.  We will achieve this by placing more point masses for better approximations.

        \paragraph{Construction of $\lambda$}
          First construct $\lambda_1$. By Kronecker's theorem, we can find $t_1^{1,1},\dots,t_1^{N,1}$ and corresponding integers $q$ (not all the same) such that 
            $$|-t_1^{j,1}\log p_1 -\theta_1^j  -2\pi q|<2^{-1}; \text{ for } j=1,\dots,N.$$
          Repeat this to find $t_1^{j,2}>\max\limits_j t_1^{j,1}$ so that there are two point masses corresponding to each component of $\mu$. (In future steps, we will repeat this so that there are $2^k\cdot\|\lambda_{k-1}\|$ point masses for each component.) Now define
            $$ \lambda_1 = \sum_{j=1}^{N}c_j \left(\delta_{t_1^{j,1}} + \delta_{t_1^{j,2}}\right) $$
          and note that $\|\lambda_1\| = 2$.\\
          Inductively construct a sequence of measures $\{\lambda_k\}_{k=2}^\infty$:

          For $k>1$ choose $T_{k-1}> \max\{t_{k-1}^{j,m}; j=1,\dots,N,\, m=1,\dots,2^{k-1}\}$. Again using Kronecker's theorem, find points $\{t_k^{j,1}\}_{j=1}^N>T_{k-1}$ and corresponding integers $q$ such that
            \begin{equation} \label{star}
              |-t_k^{j,1}\log p_r -\theta_r^j  -2\pi q|<2^{-k}; \text{ for } j=1,\dots,N \text{ and } r=1,\dots,k.
            \end{equation}
          This means that this inequality holds for all $j$ and for the first $k$ primes (or, equivalently, the first $k$ coordinates in $\T^\infty$.) Repeat this to find $N$ more points $\{t_k^{j,2}\}_{j=1}^N$ that satisfy (\ref{star}) and such that $t_k^{j,2}>\max\limits_j t_k^{j,1}$. Continue until there are $M_k=2^k\cdot \|\lambda_{k-1}\|$ points for each $\omega_j$. Define
          \begin{align*}
            \gamma_1 &= \lambda_1\\
            \gamma_k &= \sum_{m=1}^{M_k} \sum_{j=1}^{N}c_j \delta_{t_k^{j,m}}, \quad \|\gamma_k\|=2^k\|\lambda_{k-1}\|
          \end{align*}
          and
          \begin{align}
            \lambda_k &= \lambda_{k-1}+\gamma_k = \sum_{\ell=1}^{k}\gamma_\ell\\
            \|\lambda_k\|&=\lambda_k\left([0,T_k]\right)=(2^k+1)\|\lambda_{k-1}\| \label{sizeofLambdaK}
          \end{align}
          and then let $\lambda = \sum_{\ell=1}^{\infty}\gamma_\ell$. Note that for any $T$ there is some $k$ such that $\lambda\left([0,T]\right)=\lambda_k\left([0,T]\right)$.

        \paragraph{Proof that $\lambda$ satisfies (\ref{nugood})}
          Now we will verify that this measure gives the correct limit. We will use the continuity of $|F|^2$ as in (\ref{continuitything}). Given $\epsilon>0$, there exists $\delta_j>0$ such that $\|\omega_j-z\|_{\T^d}<\delta_j\Rightarrow \left||F(\omega_j)|^2-|F(z)|^2\right|<\epsilon$. Now choose $\delta=\min\limits_j \delta_j$ so 
          \begin{equation}\label{cont}
            \|\omega_j-z\|_{\T^d}<\delta \Rightarrow \left||F(\omega_j)|^2-|F(z)|^2\right|<\epsilon\quad \forall j.
          \end{equation}

          Using (\ref{star}) and the fact that the length of a chord of a circle can be bounded by the corresponding part of the circumference, for $z_{k,r}^{j,m} = e^{-it_{k}^{j,m}\log p_r}$, and for large $k$, we have
          \begin{equation}
            |\omega_{j,r}-z_{k,r}^{j,m}|< 2\cdot 2^{-k} = 2^{-k+1}
          \end{equation}
          and
          \begin{align*}
            \|\omega_j-z_k^{j,m}\|_{\T^d}^2 
            &= |\omega_{j,1}-z_{k,1}^{j,m}|^2 +\cdots + |\omega_{j,d}-z_{k,d}^{j,m}|^2\\
            &= (2^{-k+1})^2\cdot d.
          \end{align*}

          For every $T$ there is some $k$ such that $T\in [T_k, T_{k+1}]$, so choose $T$ large enough that 
            $$(2^{-k+1})^2\cdot d < \delta^2.$$
          So for point masses $t^{j,m}\in \text{supp}\, \lambda \cap [T_{k-1},\infty)$, (\ref{star}) holds for all $j,m$. (Here we omit the subscript $k$ because the estimate works for every point mass in $[T_{k-1},\infty)$.) Rewriting the continuity argument (\ref{cont}) using the Bohr lift yields
          \begin{equation} \label{tosubEpsilon}
            \left||F(\omega_j)|^2-|f(it^{j,m})|^2\right|<\epsilon \quad \forall t^{j,m}\in \text{supp}\, \lambda \cap [T_{k-1},\infty).
          \end{equation}
          For this $T$, consider 
          \begin{align*}
            &\left|\frac{1}{\lambda([0,T])} \int_{0}^{T} |f(it)|^2 d\lambda - \int_{\T^{\infty}}|F(z)|^2d\mu \right| \\
              &\quad \leq \frac{\lambda_{k-1}\left(\left[0,T_{k-1}\right]\right)}{\lambda_k([0,T_k])} \|f\|^2_\infty 
                + \left|\frac{1}{\lambda([0,T])} \int_{T_{k-1}}^{T} |f(it)|^2 d\lambda - \int_{\T^{\infty}}|F(z)|^2d\mu \right|\\
              &\quad = \frac{1}{2^k+1} \|f\|^2_\infty 
                + \left|\frac{1}{\lambda([0,T])} \int_{T_{k-1}}^{T} |f(it)|^2 d\lambda - \int_{\T^{\infty}}|F(z)|^2d\mu \right|.
          \end{align*}
          The norm $\|f\|^2_\infty$ is bounded, so the first term goes to zero as $T$ (and therefore $k$) goes to infinity, so we only need to consider the second term:
          \begin{align*}
            &\left|\frac{1}{\lambda([0,T])} \int_{T_{k-1}}^{T} |f(it)|^2 d\lambda - \int_{\T^{\infty}}|F(z)|^2d\mu \right|\\
              &\quad=\left|\frac{1}{\lambda([0,T])} \left[\int_{T_{k-1}}^{T_k} |f(it)|^2 d\lambda +\int_{T_{k}}^{T} |f(it)|^2 d\lambda\right]
                  - \int_{\T^{\infty}}|F(z)|^2d\mu \right|.\\
            \intertext{Evaluate the integrals using the definition of $\lambda$ and letting $X_j = \{t_{k+1}^{j,m}\} \in [T_k, T]\cap \text{supp}\,\lambda\}$. (Note that $\lambda[T_k,T] = \sum_{j=1}^{N}c_j|X_j|$.)}
                &\quad=\left|\frac{1}{\lambda([0,T])} \left[\sum_{m=1}^{M_k}\sum_{j=1}^{N}c_j|f(it_k^{j,m})|^2 
                  +\sum_{j=1}^{N}c_j\sum_{t\in X_j}^{}|f(it)|^2\right]
                  - \int_{\T^{\infty}}|F(z)|^2d\mu \right|.\\
            \intertext{Add and subtract $|F(\omega_j)|^2$ appropriately, rearrange, and use the triangle inequality to get}
                &\quad\leq \frac{1}{\lambda([0,T])} \left[\sum_{m=1}^{M_k}\sum_{j=1}^{N}c_j\left||f(it_k^{j,m})|^2 - |F(\omega_j)|^2\right| +\sum_{j=1}^{N}c_j\sum_{t\in X_j}\left||f(it)|^2- |F(\omega_j)|^2\right|\right]\\
                  &\qquad \qquad+\left|\frac{2^k \|\lambda_{k-1}\|}{\lambda([0,T])}\int_{\T^{\infty}}|F(z)|^2d\mu - \int_{\T^{\infty}}|F(z)|^2d\mu +\frac{1}{\lambda([0,T])}\sum_{j=1}^{N}c_j\sum_{t\in X_j}|F(\omega_j)|^2\right|.\\
              \intertext{$T$ is large enough that the continuity condition (\ref{tosubEpsilon}) holds, so}
                &\quad < \frac{1}{\lambda([0,T])} \left[\sum_{m=1}^{M_k}\sum_{j=1}^{N}c_j\epsilon +\sum_{j=1}^{N}c_j\sum_{t\in X_j}\epsilon\right]\\
                  &\qquad \qquad+\left|\frac{2^k \|\lambda_{k-1}\|}{\lambda([0,T])}\int_{\T^{\infty}}|F(z)|^2d\mu - \int_{\T^{\infty}}|F(z)|^2d\mu +\frac{1}{\lambda([0,T])}\sum_{j=1}^{N}c_j\sum_{t\in X_j}|F(\omega_j)|^2\right|\\
                &\quad = \frac{(2^k)\|\lambda_{k-1}\| +\lambda[T_k,T]}{(2^k+1)\|\lambda_{k-1}\| +\lambda[T_k,T]} \epsilon\\
                  &\qquad \qquad+\left|\left(\frac{2^k \|\lambda_{k-1}\|}{\lambda([0,T])} - 1 \right)\int_{\T^{\infty}}|F(z)|^2d\mu
                  +\frac{1}{\lambda([0,T])}\sum_{j=1}^{N}c_j|X_j||F(\omega_j)|^2\right|.
          \end{align*}
          For large $k$, the first term is small, as needed. For the second term we consider three cases depending on the size of the sets $X_j$. When we constructed $\lambda$, we placed the point masses in sets of size $N$ so that each mass $\omega_j$ had a representative, and then we repeated this. This means that $\left||X_{j_\ell}|-|X_{j_i}|\right|\leq1$ for $i\neq\ell$, so we can consider the cases 
          \begin{enumerate}
             \item $|X_j|=0$ for all $j$,
             \item $|X_j|=C $ for all $j$, and
             \item $|X_j|=C$ for $j=1,\dots,J$ and $|X_j|=C+1$ for $j=J+1,\dots,N$.
           \end{enumerate}
           
          \paragraph{Case 1: $|X_j|=0$ for all $j$} In this case, $\frac{1}{\lambda([0,T])}\sum_{j=1}^{N}c_j|X_j||F(\omega_j)|^2=0$ and $\lambda([0,T])= \lambda([0,T_k]) = (2^k+1)\|\lambda_{k-1}\|$, so we have
          \begin{equation*}
            \left|\frac{2^k \|\lambda_{k-1}\|}{\lambda([0,T])} - 1 \right|\int_{\T^{\infty}}|F(z)|^2d\mu
            = \left|\frac{2^k}{2^k+1} - 1 \right|\int_{\T^{\infty}}|F(z)|^2d\mu
          \end{equation*}
          which is small for large $k$.
          \paragraph{Case 2: $|X_j|=C\leq M_{k+1}$ for all $j$} In this case $\lambda([0,T])= \lambda([0,T_k]) +\sum_{j=1}^{N}c_j\cdot C = (2^k+1)\|\lambda_{k-1}\| +C$. Also, note that $\sum_{j=1}^{N}c_j |X_j||F(\omega_j)|^2 = C\int_{\T^{\infty}}|F(z)|^2d\mu$. Then, substituting and simplifying gives
          \begin{align*}
            &\left|\left(\frac{2^k \|\lambda_{k-1}\|}{\lambda([0,T])} - 1 \right)\int_{\T^{\infty}}|F(z)|^2d\mu
                 +\frac{1}{\lambda([0,T])}\sum_{j=1}^{N}c_j |X_j||F(\omega_j)|^2\right|\\
              &\quad= \left|\frac{2^k \|\lambda_{k-1}\|+C}{(2^k+1)\|\lambda_{k-1}\| +C} - 1 \right|\int_{\T^{\infty}}|F(z)|^2d\mu
          \end{align*}
          and this is small for large $k$.
          \paragraph{Case 3: $|X_j|=C$ for $j=1,\dots,J$ and $|X_j|=C+1$ for $j=J+1,\dots,N$} Similarly to the previous case, we get
          \begin{align*}
            &\left|\left(\frac{2^k \|\lambda_{k-1}\|}{\lambda([0,T])} - 1 \right)\int_{\T^{\infty}}|F(z)|^2d\mu
                 +\frac{1}{\lambda([0,T])}\sum_{j=1}^{N}c_j |X_j||F(\omega_j)|^2\right|\\
              &\quad=\left|\left(\frac{2^k \|\lambda_{k-1}\|}{\lambda([0,T])} - 1 \right)\int_{\T^{\infty}}|F(z)|^2d\mu
                 +\frac{C}{\lambda([0,T])}\sum_{j=1}^{N}c_j|F(\omega_j)|^2 
                 +\frac{1}{\lambda([0,T])}\sum_{j=J+1}^{N}c_j|F(\omega_j)|^2\right|\\
              &\quad\leq \left|\frac{2^k \|\lambda_{k-1}\|+C}{(2^k+1)\|\lambda_{k-1}\| +C+ \sum_{j=J+1}^{N}c_j} - 1 \right|\int_{\T^{\infty}}|F(z)|^2d\mu
                +\frac{1}{\lambda([0,T])}\sum_{j=J+1}^{N}c_j|F(\omega_j)|^2.
          \end{align*}
          The first term is small as in Case 2, and in the second term $\frac{1}{\lambda([0,T])}$ is multiplied by a bounded quantity, and so this will be small for large $T$.

          In each case, 
            $$\left|\frac{1}{\lambda([0,T])} \int_{0}^{T} |f(it)|^2 d\lambda - \int_{\T^{\infty}}|F(z)|^2d\mu \right| $$
          is small for large $T$, and so the proof is complete.
      \end{proof}

      This construction also gives us the following lemma about the tail of the integral, roughly saying that the approximation error occurs near zero.
      \begin{lemma}\label{Tepsilon}
        If $\{\mu_n\}$ is a finite collection of measures $d\mu = \sum_{j=1}^{J_n} c_j^n \delta_{\omega_j^n}$ with $\sum_{j=1}^{J_n} c_j^n = 1$ and $\{F_m\}_{m=1}^M$ is a finite set of polynomials with corresponding Dirichlet polynomials $\{f_m\}$. If $\lambda_n$ is the measure constructed as in Lemma \ref{lincombo} corresponding to $\mu_n$, then given $\epsilon>0$, there exists $T_\epsilon, T'$ such that for $T>T'$
          $$\left|\frac{1}{\lambda_n([T_\epsilon,T])} \int_{T_\epsilon}^{T} |f_m(it)|^2 d\lambda_n-\int_{\T^{\infty}}|F_m(z)|^2d\mu_n\right|<\epsilon,\, \forall m,n.$$
      \end{lemma}
        The proof is very similar to the proof of Lemma \ref{lincombo}, with the slight difference that we are also choosing a $T'$. This is necessary because we need $\lambda_n([T_\epsilon,T])$ to be large, and because we need each point mass of each $\mu_n$ to be represented by at least one point mass in our interval. There always is some sufficiently large $T'$, but even for very large $T_\epsilon$, one must still have a large interval $(T_\epsilon, T)$. (For example, the estimate may not hold for $(T_\epsilon, T_\epsilon +1)$.)

        One way to ensure that $T'$ is large enough is to require that each $\omega_j^n$ has many more than $\left\lceil \sum_{\mu_n} \sum_{j=1}^{{J_n}} |c_j| \right\rceil$ representatives $t_\ell^j$ in $[T_\epsilon, T']$ (i.e. such that $|X_j|=|\{t_\ell^j\in [T_\epsilon, T']\}|\gg \left\lceil \sum_{\mu_n} \sum_{j=1}^{{J_n}} |c_j| \right\rceil$, for each $j$).
      \begin{remark}
        If $T^{*}>T_\epsilon$,  $T'$ can be found so that the estimate of the lemma holds on $[T^{*},T]$ for $T>T'$. Roughly, this says that the estimate doesn't get any worse, as long as the intervals are sufficiently long.
      \end{remark}

\section{Proof of Theorem \ref{mainthm} part \ref{hardway}} 
  \label{secproof_of_theorem_mainthm}
    We may now return to the main result.
    \begin{mythma} 
      Let \(\mu\) be a Borel probability measure on the infinite torus \(\T^{\infty}\). There exists a locally finite Borel measure \(\lambda\) on $\R$, such that, for all $f \in \cA(\cplus)$
      \begin{equation} \tag{\ref{bigequality}}
        \lim_{T \to \infty} \frac{1}{\lambda([0,T])} \int_{0}^{T}\left|f(it)\right|^2 d\lambda(t)
        = \int_{\T^{\infty}} \left| F(z)\right|^2 d\mu(z).
      \end{equation} 
    \end{mythma}
    \begin{proof}
      Let $\{F_m\}_{m=1}^{\infty}$ be a countable set of polynomials which is dense in $A(\D^\N)$. (By definition, the polynomials are dense in $A(\D^\N)$, and there is a countable dense set of polynomials within each $A(\D^d)$ for finite $d$, so use Cantor diagonalization.) We only need to prove the theorem for $F_m$ (corresponding to a Dirichlet polynomial $f_m$) in this dense set: for $F\in A(\D^\N)$, there is some $F_m$ such that $\sup\left||F_m|^2-|F|^2\right|<\epsilon$ (and similarly $\sup\left||f_m|^2-|f|^2\right|<\epsilon$.) So 
      \begin{align*}
        \left|\,\int_{\T^{\infty}}|F_m(z)|^2d\mu -\int_{\T^{\infty}}|F(z)|^2d\mu \right| &< \epsilon\\
        \intertext{and}
        \left|\frac{1}{\lambda[0,T]}\int_{0}^{T}|f_m(it)|^2d\lambda - \frac{1}{\lambda[0,T]}\int_{0}^{T}|f(it)|^2d\lambda\right| &< \epsilon
      \end{align*}
      for any measure $\lambda$ and for all $T$. So now we have
      \begin{align*}
        &\left|\frac{1}{\lambda[0,T]}\int_{0}^{T}|f(it)|^2d\lambda -\int_{\T^{\infty}}|F(z)|^2d\mu\right|\\
          &\quad\leq \left|\frac{1}{\lambda[0,T]}\int_{0}^{T}|f_m(it)|^2d\lambda - \frac{1}{\lambda[0,T]}\int_{0}^{T}|f(it)|^2d\lambda\right| 
            +\left|\frac{1}{\lambda[0,T]}\int_{0}^{T}|f_m(it)|^2d\lambda - \int_{\T^{\infty}}|F_m(z)|^2d\mu\right|\\
            &\qquad \qquad+\left|\,\int_{\T^{\infty}}|F_m(z)|^2d\mu -\int_{\T^{\infty}}|F(z)|^2d\mu \right|\\
          &\quad< \epsilon +\left|\frac{1}{\lambda[0,T]}\int_{0}^{T}|f_m(it)|^2d\lambda - \int_{\T^{\infty}}|F_m(z)|^2d\mu\right| + \epsilon.
      \end{align*}
      It remains to show that 
      \begin{equation} \label{equalityforborel}
        \left|\frac{1}{\lambda[0,T]}\int_{0}^{T}|f_m(it)|^2d\lambda - \int_{\T^{\infty}}|F_m(z)|^2d\mu\right|<\epsilon.
      \end{equation} 
      By Lemma \ref{dense}, there exists a sequence of linear combinations of point masses $\{\mu_n\}$ that converges weak-$\ast$ to $\mu$. Choose $k$ subsequences $\{\mu_{n_{j,k}}\}$ such that for $m=1,\dots ,2^k$,
      \begin{equation*}
        \left|\,\int_{\T^{\infty}}|F_m(z)|^2d\mu_{n_{j,k}} -\int_{\T^{\infty}}|F_m(z)|^2d\mu\right|<2^{-j}
      \end{equation*}
      and re-index so that $\{\mu_{j,k}\}=\{\mu_{n_{j,k}}\}$. Now, given $m$, for $k\geq \lceil \log_2 m \rceil$, we have control over the convergence:
      \begin{equation}\label{goodmuseq}
        \left|\,\int_{\T^{\infty}}|F_m(z)|^2d\mu_{j,k} -\int_{\T^{\infty}}|F_m(z)|^2d\mu\right|<2^{-j}.
      \end{equation}
      Also, for each of these $\mu_{j,k}$, there is a corresponding $\lambda_{j,k}$ as constructed above that satisfies (\ref{nugood}) for all Dirichlet polynomials, and in particular, works for all $F_m$. 

      \paragraph{Construction of $\lambda$} We construct $\lambda$ using a process similar to the linear combination case: we will find an approximation and then repeat it so that better approximations appear more. However, this case, we will not be approximating with point masses, but with the measures $\lambda_j$ constructed as in the previous case using the Lemma \ref{lincombo}. 

      By Lemma \ref{Tepsilon}, there is some $T_1$ such that for $j,m =1,2$ and for sufficiently large $T_1^{(1)}$
      \begin{equation}\label{star1}
        \left|\frac{1}{\lambda[T_1,T_1^{(1)}]}\int_{T_1}^{T_1^{(1)}}|f_m(it)|^2d\lambda_j - \int_{\T^{\infty}}|F_m(z)|^2d\mu_{j,k}\right|<\frac{1}{2}.
      \end{equation}
      Define
      $$\lambda^{(1)}=\gamma_1^{(1)} = \sum_{j=1}^{2}\frac{\lambda_j|_{[T_1, T_1^{(1)}]}}{\lambda_j([T_1, T_1^{(1)}])} $$
      Note that $\text{supp}\, \lambda^{(1)} = [T_1, T_1^{(1)}]$ and $\|\lambda^{(1)}\|=2.$ 

      Now find $T_2\geq T_1^{(1)}$ such that for $j,m=1, \dots, 4$, and for sufficiently large $T$,
      \begin{equation}\label{starstar}
        \left|\frac{1}{\lambda[T_2,T]}\int_{T_2}^{T}|f_m(it)|^2d\lambda_j - \int_{\T^{\infty}}|F_m(z)|^2d\mu_{j,k}\right|<\frac{1}{2^2}.
      \end{equation}
      Let $T_2^{(1)}$ be large enough that (\ref{starstar}) holds. Now we repeat the approximation: by the remark after the proof of Lemma \ref{Tepsilon}, there is some $T_2^{(2)}$ large enough so that for $j,m = 1,\dots, 4$
      \begin{equation*}
        \left|\frac{1}{\lambda[T_2^{(1)},T_2^{(2)}]}\int_{T_2^{(1)}}^{T_2^{(2)}}|f_m(it)|^2d\lambda_j - \int_{\T^{\infty}}|F_m(z)|^2d\mu_{j,k}\right|<\frac{1}{2^2}.
      \end{equation*}
      Define 
      \begin{equation*}
        \gamma_2^{(\ell)} = \sum_{j=1}^{2^2}\frac{\lambda_j|_{[T_2^{(\ell-1)},T_2^{(\ell)}]}}{\lambda_j([T_2^{(\ell-1)},T_2^{(\ell)}])},\qquad T_2^{(0)} = T_2
      \end{equation*}
      and 
        $$\lambda^{(2)}= \lambda^{(1)}+ \sum_{\ell =1}^{\|\lambda^{(1)}\|=2}\gamma_2^{(\ell)}.$$ 
      Note that $\|\gamma_2^{(\ell)}\|=2^2$ and $\|\lambda^{(2)}\| = \|\lambda^{(1)}\|+2\cdot \|\gamma_2^{(\ell)}\| $. 

      Continue this process, at level $k$ finding $T_k\geq T_{k-1}^{(\|\lambda^{(k-2)}\|)} $ such that for $j,m=1, \dots 2^k$ and for sufficiently large $T$
      \begin{equation}\label{starK}
        \left|\frac{1}{\lambda[T_k,T]}\int_{T_k}^{T}|f_m(it)|^2d\lambda_j - \int_{\T^{\infty}}|F_m(z)|^2d\mu_{j,k}\right|<\frac{1}{2^k}.
      \end{equation}
      Let $T_k^{(1)}=T$ be large enough that (\ref{starK}) holds and find $T_k^{(2)}$ such that
      \begin{equation*}
        \left|\frac{1}{\lambda[T_k^{(1)},T_k^{(2)}]}\int_{T_k^{(1)}}^{T_k^{(2)}}|f_m(it)|^2d\lambda_j - \int_{\T^{\infty}}|F_m(z)|^2d\mu_{j,k}\right|<\frac{1}{2^k}.
      \end{equation*}
      Repeat this to get an increasing sequence $\{T_k^{(\ell)}\}_{\ell=0}^{\|\lambda_{k-1}\|}$ (where $T_k^{(0)}=T_k$) such that for $j,m=1,\dots,2^k$ and for $\ell=1,\dots,\|\lambda^{(k-1)}\|$,
      \begin{equation}\label{Tkell}
        \left|\frac{1}{\lambda[T_k^{(\ell-1)},T_k^{(\ell)}]}\int_{T_k^{(\ell-1)}}^{T_k^{(\ell)}}|f_m(it)|^2d\lambda_j - \int_{\T^{\infty}}|F_m(z)|^2d\mu_{j,k}\right|<\frac{1}{2^k}.
      \end{equation}
      Define for $\ell=1,\dots,\|\lambda^{(k-1)}\|$
      \begin{equation*}
        \gamma_k^{(\ell)} = \sum_{j=1}^{2^k}\frac{\lambda_j|_{[T_k^{(\ell-1)},T_k^{(\ell)}]}}{\lambda_j([T_k^{(\ell-1)},T_k^{(\ell)}])},\qquad \|\gamma_k^{(\ell)}\|=2^k
      \end{equation*}
      and 
        $$\lambda^{(k)}= \lambda^{(k-1)}+ \sum_{\ell =1}^{\|\lambda^{(k-1)}\|}\gamma_k^{(\ell)}.$$
      Letting $\lambda = \lim\limits_{k\to\infty} \lambda^{(k)}$, we have that $\lambda[0, T_{k+1}] = \|\lambda^{(k)}\| = (2^k+1)\|\lambda^{(k-1)}\|= (2^k+1)\lambda[0,T_k]$.
      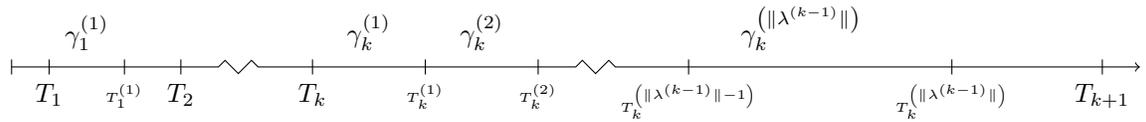
\begin{figure}[H]\label{fig:BorelDistr}
        \centering
        \begin{tikzpicture}
  \draw[->,decoration=zigzag] (0,0) -- (2.75, 0) decorate {--(3.4, 0)} --(7.5,0) decorate {--(8.15,0)}-- (15,0);
  \foreach \i in {0,.5,1.5,2.25,4,5.5,7,9, 12.5, 14.5} 
  {
    \draw (\i,0.1) -- + (0,-0.2) ;
  }
  \node[below] at (.5,-.1) {$T_1$};
    \node[above] at (1,.1) {$\gamma_1^{(1)}$};
    \node[below] at (1.5,-.1) {\tiny$T_1^{(1)}$};
  \node[below] at (2.25,-.1) {$T_2$};

  \node[below] at (4,-.1) {$T_k$};
    \node[above] at (4.75,.1) {$\gamma_k^{(1)}$};
    \node[below] at (5.5,-.1) {\tiny$T_k^{(1)}$};
    \node[above] at (6.25,.1) {$\gamma_k^{(2)}$};
    \node[below] at (7,-.1) {\tiny$T_k^{(2)}$};

    \node[below] at (9,-.1){\tiny$T_k^{\left(\|\lambda^{(k-1)}\|-1\right)}$};
    \node[above] at (10.5,.1) {$\gamma_k^{\left(\|\lambda^{(k-1)}\|\right)}$};
    \node[below] at (12.5,-.1) {\tiny$T_k^{\left(\|\lambda^{(k-1)}\|\right)}$};

  \node[below] at (14.5,-.1) {$T_{k+1}$};
\end{tikzpicture}
        \caption[Construction of the Borel measure]{At each level $k$ the small intervals give the same quality of estimate, as in \eqref{Tkell}. The repetition here serves the same purpose as in the point mass case: better estimates contribute more to mass of the measure over the real line.}
      \end{figure}
      \paragraph{Proof that $\lambda$ satisfies (\ref{equalityforborel})} 
      For any $T$ there is some $k$ such that $T\in [T_{k+1}, T_{k+2}]$. Choose $T$ large enough that $2^{k}\geq m$. From here, consider three cases: where $T>T_{k+1}^{(\|\lambda^{(k)}\|)}$, where $T<T_{k+1}^{(1)}$, and where $T\in [T_{k+1}^{(q)}, T_{k+1}^{(q+1)}]$ for some $1\leq q < \|\lambda^{(k)}\| $.

      Now, examine
      \begin{align*}
        &\left|\frac{1}{\lambda[0,T]} \int_{0}^{T} |f_m(it)|^2 d\lambda-\int_{\T^{\infty}}|F_m(z)|^2d\mu\right|\nonumber\\
          &\quad\leq \left|\frac{1}{\lambda[0,T]} \int_{0}^{T_{k}} |f_m(it)|^2 d\lambda\right| + \left|\frac{1}{\lambda[0,T]} \int_{T_{k}}^{T} |f_m(it)|^2 d\lambda-\int_{\T^{\infty}}|F_m(z)|^2d\mu\right|\nonumber\\
          &\quad \leq \frac{1}{2^k+1}\|f_m\|^2_\infty + \left|\frac{1}{\lambda[0,T]} \int_{T_{k}}^{T_{k+1}} |f_m(it)|^2 d\lambda + \frac{1}{\lambda[0,T]} \int_{T_{k+1}}^{T} |f_m(it)|^2 d\lambda-\int_{\T^{\infty}}|F_m(z)|^2d\mu\right|\nonumber\\
      \end{align*}
      The first term here is small for large $T$, so consider the second term, using the definition of $\lambda$ and adding and subtracting $\int|F_m|^2d\mu_{j,k}$ and $\int|F_m|^2d\mu$ where appropriate
      \begin{align}
        &\left|\frac{1}{\lambda[0,T]} \int_{T_{k}}^{T_{k+1}} |f_m(it)|^2 d\lambda-\int_{\T^{\infty}}|F_m(z)|^2d\mu\right|\nonumber\\
          &\quad\leq \frac{1}{\lambda[0,T]} \sum_{\ell=1}^{\|\lambda^{(k-1)}\|}\sum_{j=1}^{2^{k}}\left|\frac{1}{\lambda_{j,k}[T_{k}^{\ell-1},T_{k}^{\ell}]}\int_{T_{k}^{\ell-1}}^{T_{k}^{\ell}}|f_m|^2d\lambda_{j,k} - \int_{\T^\infty}^{}|F_m|^2d\mu_{j,k}\right|\nonumber\\
            &\qquad \qquad+ \frac{1}{\lambda[0,T]} \sum_{\ell=1}^{\|\lambda^{(k-1)}\|}\sum_{j=1}^{2^{k}}\left|\,\int_{\T^{\infty}}^{}|F_m|^2d\mu_{j,k} -\int_{\T^\infty}^{}|F_m|^2d\mu\right|\nonumber\\ 
            &\qquad \qquad+ \left|\frac{1}{\lambda[0,T]} \int_{T_{k+1}}^{T} |f_m(it)|^2 d\lambda + \left(\frac{2^k\|\lambda^{(k-1)}\|}{\lambda[0,T]} -1\right)\int_{\T^{\infty}}|F_m(z)|^2d\mu\right|. \nonumber\\
        \intertext{Applying (\ref{Tkell}) then gives}
          &\quad < \frac{\|\lambda^{(k-1)}\|}{\lambda[0,T]}\sum_{j=1}^{2^{k}}2^{-k} + \frac{\|\lambda^{(k-1)}\|}{\lambda[0,T]} \sum_{j=1}^{2^{k}}\left|\,\int_{\T^{\infty}}^{}|F_m|^2d\mu_{j,k} -\int_{\T^\infty}^{}|F_m|^2d\mu\right| \nonumber\\ 
            &\qquad \qquad+ \left|\frac{1}{\lambda[0,T]} \int_{T_{k+1}}^{T} |f_m(it)|^2 d\lambda + \left(\frac{2^k\|\lambda^{(k-1)}\|}{\lambda[0,T]} -1\right)\int_{\T^{\infty}}|F_m(z)|^2d\mu\right|. \nonumber\\
        \intertext{Note that $\lambda[0,T] > \lambda[0,T_{k+1}] = (2^k+1)\|\lambda^{(k-1)}\|$. We can get control over the second term by recalling how we chose our sequence of $\mu_{j,k}$ as in \eqref{goodmuseq} because $m\leq 2^{k}$:}
          &\quad < \frac{1}{2^{k}+1}\sum_{j=1}^{2^{k}}2^{-k} + \frac{1}{2^{k}+1} \sum_{j=1}^{2^{k}}2^{-j} \nonumber\\ 
            &\qquad \qquad+ \left|\frac{1}{\lambda[0,T]} \int_{T_{k+1}}^{T} |f_m(it)|^2 d\lambda + \left(\frac{2^k\|\lambda^{(k-1)}\|}{\lambda[0,T]} -1\right)\int_{\T^{\infty}}|F_m(z)|^2d\mu\right|. \label{reduce1and3}\\
        \intertext{Now, because $\sum_{j=1}^{2^{k}}2^{-k}=1$ and $\sum_{j=1}^{2^{k}}2^{-j}<1$, we have}
          &\quad < \frac{2}{2^k+1} + \left|\frac{1}{\lambda[0,T]} \int_{T_{k+1}}^{T} |f_m(it)|^2 d\lambda + \left(\frac{2^k\|\lambda^{(k-1)}\|}{\lambda[0,T]} -1\right)\int_{\T^{\infty}}|F_m(z)|^2d\mu\right|. \label{goodfor1}
      \end{align}
      The first term is obviously small for large $k$, so we only need to consider the term in absolute values. Similarly to the above proof for point masses, we must now consider two cases: where $T$ falls after the repetitions, and where $T$ falls within one of the repetitions:
        \begin{enumerate}
          \item $T>T_{k+1}^{(\|\lambda^{(k)}\|)}$ so that $T$ occurs after all the repetitions. 
          \item $T_{k+1}^{(q)}\leq T \leq T_{k+1}^{(q+1)}$ for some $0\leq q < \|\lambda^{(k)}\| $ so that $T$ falls within a repetition.
        \end{enumerate}
      For the first case, note that $\lambda$ is not supported on $[T_{k+1}^{(\|\lambda^{(k)}\|)}, T_{k+2}] $, so this is equivalent to $T=T_{k+2}.$ It is not hard to see that the proof for this case is equivalent to the proof if $T=T_{k+1}$. In \ref{goodfor1}, setting $T=T_{k+1}$ forces the integral to be $0$, and the second term is small for large $k$ because $\lambda[0,T]=\lambda[0,T_{k+1}] = (2^k +1)\|\lambda^{(k-1)}\|$.

      For the second case, $\lambda[T_k+1, T]$ consists of the first $q$ subintervals, and then has a contribution from $\left[T_{k+1}^{(q)},T_{k+1}^{(q+1)}\right]$:
        \begin{align}
          \lambda[T_{k+1}, T] &= \sum_{\ell=1}^{q}\gamma_{k+1}^{(\ell)}[T_{k+1}^{\ell-1},T_{k+1}^{\ell}]+ \gamma_{k+1}^{(q+1)}[T_{k+1}^{q},T]\label{case3lastchunk}\\
            &\leq 2^{k+1}q+2^{k+1}\nonumber
        \end{align}
      Now, expand the term in absolute values using the definition of $\lambda$ and adding and subtracting appropriate terms. (Here, we use the Dirac delta so that if $q=0$, terms (I) and (III) are zero.)
        $$\delta(q) = \begin{cases}
            0 &\quad \text{if } q=0\\
            1 &\quad \text{if } q\neq0.
          \end{cases} $$

      \begin{align*}
        &\left|\frac{1}{\lambda[0,T]}\int_{T_{k+1}}^{T} |f_m(it)|^2 d\lambda +
              \left(\frac{2^k\|\lambda^{(k-1)}\|}{\lambda[0,T]} -1\right)\int_{\T^{\infty}}|F_m(z)|^2d\mu\right|\\
          &\quad\leq \delta(q)\frac{1}{\lambda[0,T]} \sum_{\ell=0}^{q-1}\sum_{j=1}^{2^{k+1}}\left|\frac{1}{\lambda_{j,k}[T_{k+1}^{(\ell)},T_{k+1}^{(\ell+1)}]}\int_{T_{k+1}^{(\ell)}}^{T_{k+1}^{(\ell+1)}}|f_m|^2d\lambda_{j,k} - \int_{\T^\infty}^{}|F_m|^2d\mu_{j,k}\right|\tag{I}\\
            &\qquad\qquad + \frac{1}{\lambda[0,T]}\sum_{j=1}^{2^{k+1}}\frac{1}{\lambda_{j,k}[T_{k+1}^{(q)},T_{k+1}^{(q+1)}]}\int_{T_{k+1}^{(q)}}^{T} |f_m(it)|^2 d\lambda_{j,k} \tag{II}\\
            &\qquad \qquad+ \delta(q)\frac{1}{\lambda[0,T]} \sum_{\ell=0}^{q-1}\sum_{j=1}^{2^{k+1}} \left|\,\int_{\T^\infty}^{}|F_m|^2d\mu_{j,k} -\int_{\T^\infty}^{}|F_m|^2d\mu\right|\tag{III}\\
            &\qquad\qquad + \left|\frac{2^k\|\lambda^{(k-1)}\|+2^{k+1}q}{\lambda[0, T]}-1\right| \int_{\T^\infty}^{}|F_m|^2d\mu. \tag{IV} \\
      \end{align*}
      \begin{align*}
        \intertext{Using \eqref{case3lastchunk}, $\frac{1}{\lambda[0,T]}>\frac{1}{\lambda[0,T_{k+1}]+2^{k+1}(q+1)} $ so (II) is bounded by $\frac{2^{k+1}}{\lambda[0,T_{k+1}]+2^{k+1}(q+1)}\left\|f_m\right\|^2_\infty$. Similarly to \eqref{reduce1and3}, (I) and (III) can be simplified using \eqref{Tkell} and \eqref{goodmuseq} yielding:} 
          \text{(I)+(III)+(II)+(IV)}
            &< \frac{q}{\lambda[0,T_{k+1}]+2^{k+1}(q+1)} \left(\sum_{j=1}^{2^{k+1}}2^{-(k+1)} + \sum_{j=1}^{2^{k}}2^{-j}\right)\tag{I+III}\\
            &\qquad \qquad + \frac{2^{k+1}}{\lambda[0,T_{k+1}]+2^{k+1}(q+1)}\left\|f_m\right\|^2_\infty\tag{II}\\
            &\qquad \qquad + \left|\frac{2^k\|\lambda^{(k-1)}\|+2^{k+1}q}{\lambda[0,T_{k+1}]+2^{k+1}(q+1)}-1\right| \int_{\T^\infty}^{}|F_m|^2d\mu.\tag{IV} \\  
      \end{align*}
      The first term here is small for large $k$ because $\left(\sum_{j=1}^{2^{k+1}}2^{-(k+1)} + \sum_{j=1}^{2^{k}}2^{-j}\right)<2$. To show that the other terms are small, consider
      \begin{equation}\label{eqrecursive}
        \lambda[0,T_{k+1}] = (2^k+1)\|\lambda^{(k-1)}\| = (2^k+1)(2^{k-1}+1)\cdots(2^2+1)\cdot 2> 2^{\frac{k(k+1)}{2}}.
      \end{equation}
      Now, use the first part of \eqref{eqrecursive} to simplify the term from (IV), note that it is the same as from (II), and use the inequality:
      \begin{align*}
        \left|\frac{2^k\|\lambda^{(k-1)}\|+2^{k+1}q}{\lambda[0,T_{k+1}]+2^{k+1}(q+1)}-1\right|
          &= \frac{2^{k+1}}{\lambda[0,T_{k+1}]+2^{k+1}(q+1)}\\
          &< \frac{2^{k+1}}{2^{\frac{k(k+1)}{2}}+2^{k+1}(q+1)}.
      \end{align*}
      This is small for large $k$, and thus both (II) and (IV) are small, as needed.

      Because $k$ depends on $T$, with $T$ large implying that $k$ is large, we've shown in each case that for large T, $\left|\frac{1}{\lambda[0,T]} \int_{0}^{T} |f_m(it)|^2 d\lambda-\int_{\T^{\infty}}|F_m(z)|^2d\mu\right|$ is as small as desired, so:
        $$\lim\limits_{T\to\infty}\frac{1}{\lambda[0,T]} \int_{0}^{T} |f_m(it)|^2 d\lambda = \int_{\T^{\infty}}|F_m(z)|^2d\mu. $$
    \end{proof}

\subsection{Proof of Theorem \ref{mainthm}, Part \textit{(ii)}} 
  \label{secconverse}
    Recall the statement of Theorem \ref{mainthm}, Part \textit{(ii)}:
    \begin{mythmb} 
        Let $\lambda$ be a locally finite Borel measure on $\R$ such that the limit on the left hand side of (\ref{bigequality}) exists and is finite for all $f \in \cA(\cplus)$. Then there exists a unique Borel probability measure $\mu$ on the infinite torus \(\T^{\infty}\) such that, for all $f \in \cA(\cplus)$, (\ref{bigequality}) holds.
        \begin{equation} \tag{\ref{bigequality}}
          \lim_{T \to \infty} \frac{1}{\lambda([0,T])} \int_{0}^{T}\left|f(it)\right|^2 d\lambda(t)
          = \int_{\T^{\infty}} \left| F(z)\right|^2 d\mu(z).
        \end{equation} 
      \end{mythmb}
      \begin{proof}
        We will appeal to the Riesz Representation Theorem to show that the left hand side of (\ref{bigequality}) can be used to define a positive bounded linear functional $\phi$ on $C(\T^{\infty},\R)$ which will give us our unique $\mu$. This $\mu$ can then easily be shown to be a probability measure. We will define this functional $\phi$ on a subalgebra of $C(\T^{\infty},\R)$ and then apply Stone-Weierstrass to show that the definition extends to the whole of $C(\T^{\infty},\R)$.

        We choose to define $\phi$ on the set 
        \begin{equation*}
          \cU = \left\{\sum_{n=1}^N a_n |F_n|^2: F_n \text{ are polynomials on } {\T^{\infty}}\right\}.
        \end{equation*}
        Let
        \begin{equation} \label{blfphi}
          \phi: |F|^2 \mapsto \lim_{T \to \infty} \frac{1}{\lambda([0,T])} \int_{0}^{T}\left|f(it)\right|^2 d\lambda(t).
        \end{equation} 
        Extending this linearly gives the definition of $\phi$ on $\cU$, and if we can show that $\cU$ is dense in $C(X,\R)$, then we can extend $\phi$ to be a linear functional on $C(\T^{\infty},\R)$. 
        To show that $\cU$ is dense, apply Stone-Weierstrass: $\cU$ clearly contains the constant functions and is a vector subspace of $C(X,\R)$, and it is easy to show that it is closed under multiplication since $|F|^2|G|^2 = |FG|^2$ holds, and distribution shows that it holds for linear combinations as well. So we have that $\cU$ is a subalgebra. 

        It remains to show that $\cU$ separates points on $\T^{\infty}$: given $z\neq w\in \T^{\infty}$ we need to find a function in $\cU$, $h$ such that $h(z)\neq h(w)$. The points $z$ and $w$ must differ in at least one coordinate, so without loss of generality, assume they differ in the first one, $z_1 \neq w_1$. Consider the linear function in one variable $P(z) = az_1 + (b+ic)$, where $a,\,b,\,c \neq 0$. Given two points, the constants can be chosen such that $h(z)=|P(z)|^2$ separates those points. 

        So we have that $\cU$ is dense in $C(\T^{\infty}, \R)$, and $\phi$ as defined in (\ref{blfphi}) is clearly linear on $\cU$. Extend $\phi$ continuously to $C(\T^{\infty}, \R)$. Because $C(\T^{\infty}, \R)$ is a Banach space and because of the assumption that the limit exists and is bounded for all elements in $\cA(\cplus)$, $\phi$ is bounded on $C(\T^{\infty}, \R)$. Then Riesz Representation gives a unique measure $\mu$ on $\T^{\infty}$ such that
        \begin{equation*}
          \phi(h) = \int_{\T^{\infty}} h(z) d\mu.
        \end{equation*}
        This $\mu$ is a probability measure because $\lim_{T \to \infty} \frac{1}{\lambda([0,T])} \int_{0}^{T}\left|1\right|^2 d\lambda(t) = 1$, as needed.
      \end{proof}

  \begin{remark}[Lebesgue measure]
  \label{subthe_case_of_lebesgue_measure}
    If $\mu$ is Lebesgue measure on $\T^\infty$, $\int_{\T^{\infty}} \left| \sum a_n z^{n} \right|^2 d\mu(z)=\sum_{n=1}^\infty |a_n|^2$. In this case, it is possible to choose $\lambda$ to be Lebesgue measure, in addition to the measure constructed in the proof. Using the convergence in $\cA(\cplus)$ then, we can see that for a sequence of Dirichlet polynomials $f_m\to f\in \cA(\cplus)$ and corresponding $F_m \to F\in A(\D^{\N})$,
    \begin{align*}
      &\left|\frac{1}{T} \int_{0}^{T}\left|f(it)\right|^2 d\lambda(t) -\int_{\T^{\infty}} \left|F(z)\right|^2 d\mu(z)\right|\\
        &\quad\leq \left|\frac{1}{T} \int_{0}^{T}\left|f(it)\right|^2 d\lambda(t) -\frac{1}{T} \int_{0}^{T}\left|f_m(it)\right|^2 d\lambda(t)\right|
          +\left|\frac{1}{T} \int_{0}^{T}\left|f(it)\right|^2 d\lambda(t) -\int_{\T^{\infty}} \left|F(z)\right|^2 d\mu(z)\right|\\
          &\qquad\qquad+\left|\int_{0}^{T}\left|F_m(it)\right|^2 d\lambda(t) - \int_{0}^{T}\left|F(it)\right|^2 d\lambda(t)\right|\\
        &\quad<\epsilon
    \end{align*}
    where the first and third terms are small because of the convergence, and the middle term is small for large $T$ because Carlson's theorem holds on the boundary $\sigma=0$ for Dirichlet polynomials. This is not hard to see by simply computing the integrals on either side.

    This example lets us see explicitly that given $\mu$ we do not have uniqueness for $\lambda$.
  \end{remark} 

\section*{Acknowledgments}
  This paper is part of my PhD thesis under the advising of John McCarthy, and I would like to thank him for his suggestions and support.

\section*{References}
\bibliography{BibDirichletSeries.bib} 

\begin{thebibliography}{1}

\bibitem{Aron2016}
{\sc Aron, R., Bayart, F., Gauthier, P.~M., Maestre, M., and Nestoridis, V.}
\newblock Dirichlet approximation and universal {Dirichlet} series.
\newblock {\em arXiv:1608.06457 [math.CV]\/} (2016).

\bibitem{Bohr1913a}
{\sc Bohr, H.}
\newblock {\"{U}}ber die {Bedeutung} der {Potenzreihen} unendlich vieler
  {Variabeln} in der {Theorie} der {Dirichletschen} {Reihen} \(\sum
  \frac{a_n}{n^s}\).
\newblock {\em Nachr. Akad. Wiss. G{\"{o}}ttingen, Math.-Phys. Kl.\/} (1913),
  441--488.

\bibitem{Carlson1922}
{\sc Carlson, F.}
\newblock Contributions \`{a} la th\'{e}orie des s\'{e}ries de {Dirichlet}.
\newblock {\em Ark. Mat. 16\/} (1922), 1--19.
\newblock Note I.

\bibitem{parthaProb}
{\sc Parthasarathy, K.}
\newblock {\em Probability Measures on Metric Spaces}.
\newblock Probability and mathematical statistics. Elsevier Science, 1967.

\bibitem{SaksmanSeip2009}
{\sc Saksman, E., and Seip, K.}
\newblock Integral means and boundary limits of {Dirichlet} series.
\newblock {\em Bull. London Math. Soc. 41\/} (2009), 411--422.

\end{thebibliography}
\bibliographystyle{acm}
\end{document}